\documentclass[reqno,12pt]{article}
\usepackage{amsfonts, amsmath, amssymb, amsthm}
\usepackage{hyperref}
\hypersetup{ colorlinks = true, urlcolor = blue, linkcolor = blue, citecolor = red }
\usepackage{color}
\usepackage{dsfont}
\usepackage[T1]{fontenc}
\usepackage{authblk}
\usepackage{graphicx,tikz,xcolor}
\usepackage{xcolor,eucal,enumerate,mathrsfs}
\usepackage[normalem]{ulem}
\usepackage{epsfig} 
\usepackage[latin1]{inputenc}
\usepackage{psfrag}          
\usepackage[ruled,vlined]{algorithm2e}
\usepackage{cancel}
\usepackage{setspace}
\usepackage{mystyle}
\usepackage{comment}
\allowdisplaybreaks

\DeclareMathOperator{\loc}{loc}

\DeclareMathOperator{\dist}{dist}

\DeclareMathOperator{\diam}{diam}

\newcommand{\R}{\mathbb{R}}
\newcommand{\norm}[1]{\left\lVert#1\right\rVert}
\newcommand{\inner}[2]{\left\langle #1, #2 \right\rangle}

\newcommand{\I}{\int\limits}

\newcommand{\ssubset}{\subset\joinrel\subset}

\title{$C^{1,\alpha}$-regularity for Schr\"odinger potentials of Shannon-entropy regularized Optimal Transport}

\author[1,2]{Sumiya Baasandorj}
\author[3]{Simone Di Marino}
\author[2,4,5,6]{Augusto Gerolin}
\affil[1]{Scuola Normale Superiore di Pisa, Piazza dei Cavalieri 7, 56126 Pisa, Italy}
\affil[2]{Department of Mathematics and Statistics, University of Ottawa}
\affil[3]{Dipartimento di Matematica, MaLGa, Universit\`a di Genova}
\affil[4]{Department of Chemistry and Biomolecular Sciences, University of Ottawa}
\affil[5]{Nexus for Quantum Technologies, University of Ottawa, Canada}
\affil[6]{Instituto de Matem\'atica Pura e Aplicada, Rio de Janeiro, Brazil}

\affil[ ]{E-mail: \texttt{sbaasand@uottawa.ca, simone.dimarino@unige.it, agerolin@uottawa.ca}}

\begin{document}

\maketitle

\vspace{-6mm}

\tableofcontents

\vspace{5mm}

\begin{abstract}
	We provide the stability of $C^{1,\alpha}$ regularity of Schr\"odinger potentials of Boltzmann-Shannon entropy regularized Optimal Transport under the classical assumption of marginals supported on bounded convex sets and whose densities are bounded above and below.
 \end{abstract}

\section{Introduction}

In this paper, we obtain Cafarelli-type regularity results for the Boltzmann-Shannon regularized Optimal Transport problem
\begin{align}
    \label{EOTP:1}
    \text{W}^2_{\varepsilon}(\mu,\nu):= \inf\limits_{\pi \in \Pi(\mu,\nu)}\I_{\R^{n}\times \R^{n}} \frac{1}{2}|x-y|^{2}\,d\pi(x,y) + 
    \varepsilon H(\pi | \mu\otimes \nu), 
\end{align}
where $\Pi(\mu,\nu)$ denotes the set of transport plans between probability measures $\mu$ and $\nu$ in $\R^n$ $(n\geq 1)$. The functional $H(\cdot | \mu\otimes \nu)$ denotes the Boltzmann-Shannon relative entropy (or Kullback-Leibler divergence) with respect to the product measure $\mu\otimes \nu$,
\begin{align*}
    H(\pi | \mu\otimes\nu):=  
		\begin{cases} 
                \displaystyle\I \log\left(\frac{d\pi}{d(\mu\otimes\nu)} \right)\,d\pi& \mbox{if } \pi \ll \mu\otimes\nu,\\ 
		      +\infty & \mbox{if } \pi \not\ll \mu\otimes\nu. 
            \end{cases}
\end{align*}

More accurately, let $\Omega_1,\Omega_2\subset \R^n$ be bounded convex open sets in $\R^n$, and  $\mu,\nu\in\mathcal{P}(\R^n)$ be probability measures such that their supports $\mathrm{spt}(\mu)\subset \Omega_1$ and $\mathrm{spt}(\nu)\subset \Omega_2$, are contained, respectively, in $\Omega_1$ and $\Omega_2$. Additionally, throughout the paper we assume that the measures $\mu$ and $\nu$ are absolutely continuous with respect to Lebesgue measure
\begin{align}
    \label{density:1}
    \mu = f(x)\,dx
    \quad\text{and}\quad
    \nu = g(y)\,dy,
\end{align}
for some measurable density functions $f,g:\R^n\to\R$ such that there exist constants $0<\lambda\leqslant \Lambda$ such that 
\begin{align}
    \label{density:2}
    \lambda \leqslant f(x) \leqslant \Lambda\quad \forall x\in\Omega_1,\quad
    \lambda \leqslant g(y) \leqslant \Lambda\quad \forall y\in \Omega_2.
\end{align}
When the variational problem \eqref{EOTP:1} admits a solution $\pi_{\varepsilon} \in \Pi(\mu,\nu)$, called \textit{entropic optimal transport plan}, it is unique and has the following Gibbs distribution~\cite{DG20, LeoSurvey}
\begin{align}
    \label{dens:phi:1}
    \frac{d\pi_{\varepsilon}}{d(\mu\otimes\nu)}(x,y) = \exp\left(\frac{\varphi_{\varepsilon}(x) + \psi_{\varepsilon}(y)-|x-y|^{2}/2}{\varepsilon} \right),
\end{align}
for two measurable functions $\varphi_{\varepsilon} : \Omega_1\rightarrow \R$ and $\psi_{\varepsilon} : \Omega_2\rightarrow \R$, called \textit{Schr\"odinger potentials}. In particular, the pair $(\varphi_{\varepsilon}, \psi_{\varepsilon})$ maximizes the dual problem of \eqref{EOTP:1} which reads as

\begin{align*}
    W_{\varepsilon}^{2}(\mu,\nu) = 
    \notag\
     \sup\limits_{\varphi\in C(\Omega_{1}),\psi\in C(\Omega_{2})} 
    \left\{ D_{\varepsilon}(\varphi,\psi)\right\}  + \varepsilon ,
\end{align*}
where 
\[
D_{\varepsilon}(\varphi,\psi) = \I_{\Omega_{1}}\varphi\,d\mu + \I_{\Omega_{2}}\psi\,d\nu -\varepsilon \I_{\Omega_{1}\times \Omega_{2}} \exp\left( \frac{\varphi(x)+\psi(y)-|x-y|^2/2}{\varepsilon} \right)\,d\mu(x)\,d\nu(y),
\]
is the dual functional, see again~\cite{DG20, LeoSurvey}. For abbreviation, let us denote in the rest of the paper 
\begin{align}
    \label{schrod:3}
    u_{\varepsilon}(x):= \frac{|x|^2}{2}-\varphi_{\varepsilon}(x)
    \quad\text{and}\quad
    v_{\varepsilon}(y):= \frac{|y|^2}{2}-\psi_{\varepsilon}(y).
\end{align}
Rewriting the feasibility condition $\pi_{\varepsilon}\in\Pi(\mu,\nu)$, we find that the functions $u_{\varepsilon} : \Omega_1\rightarrow \R$ and $v_{\varepsilon}: \Omega_2\rightarrow \R$ solve the so-called \textit{Schr\"odinger system} 
\begin{align}
    \label{schrod:2}
    \left\{
    \begin{array}{lr}
        \displaystyle \exp\left(\frac{u_{\varepsilon}(x)}{\varepsilon}\right) = \I_{\Omega_2} \exp\left(\frac{\inner{x}{y}-v_{\varepsilon}(y)}{\varepsilon}\right)\,d\nu(y) 
        &\quad \mu \text{ - a.e,} \\
        \displaystyle \exp\left(\frac{v_{\varepsilon}(y)}{\varepsilon}\right) = \I_{\Omega_1} \exp\left(\frac{\inner{x}{y}-u_{\varepsilon}(x)}{\varepsilon}\right)\,d\mu(x) 
        &\quad \nu \text{ - a.e.}
    \end{array}
    \right.
\end{align}

The main result is given by the following theorem.

\begin{teo}
    \label{thm:main}
    Let $\Omega_1, \Omega_2 \subseteq \R^d$ be two convex and bounded sets and let $\mu \in \mathcal{P}(\Omega_1), \nu \in \mathcal{P}(\Omega_2)$ satisfy assumptions \eqref{density:1}-\eqref{density:2}. Then there exists $\beta\equiv \beta(n,\lambda,\Lambda)\in (0,1)$ such that for every $\varepsilon >0$ if we consider $u_{\varepsilon} : \Omega_1\rightarrow \R$ and $v_{\varepsilon} : \Omega_2 \rightarrow \R$ Schr\"odinger potentials defined in \eqref{schrod:3} that satisfy system of equations \eqref{schrod:2} we have $u_{\varepsilon}\in C^{1,\beta}_{\loc}(\Omega_{1})$. In particular for every convex subset $\Omega_{1}'\ssubset \Omega_1$, there exists a universal constant $C_{0}$ depending on $\Omega_{1}'$ and the modulus of $u_0$ (independent of $\varepsilon$) such that 
    \begin{align*}
        \sup\limits_{x\neq y \in \Omega_{1}'} \frac{|\nabla u_{\varepsilon}(x)-\nabla u_{\varepsilon}(y)|}{|x-y|^{\beta}} \leqslant C_0
    \end{align*}
    holds whenever $\varepsilon>0$. 
\end{teo}

Theorem \ref{thm:main} corresponds to the entropic version of the celebrated   Caffarelli's $C^{1,\alpha}_{loc}$-regularity theory for Monge-Amp\`{e}re equation ~\cite{Ca2, Ca3}. Although the Schr\"odinger potentials $u_{\varepsilon} \in C^{\infty}(\Omega_{1})$ and $v_{\varepsilon} \in C^{\infty}(\Omega_{2})$ are smooth for every $\varepsilon > 0$~\cite{DG20, genevay18sample}, the known general quantitative estimates beyond the first order explode when $\varepsilon \to 0^{+}$ (see for example \cite{genevay18sample} where one can obtain $\|u_{\varepsilon}^{(k)}\|_{\infty} \leqslant \frac C{\varepsilon^{k-1}}$ where $C$ does not depend on $\varepsilon$). 

 In particular proving that there exists $\beta \in (0,1)$ such that  $u_{\varepsilon}, v_{\varepsilon} \in C^{1,\beta}_{loc}$ \emph{uniformly} in $\varepsilon$ is not at all obvious. Addressing this question is one of the main motivations in this paper.

\begin{oss} Following the proof, the explicit value of $\beta$ that we can get is $\beta= \inf \{ \frac 1{n^2}, \frac{\alpha^2}{(1+\alpha)^2} \}$, where $\alpha$ is such that $u_0 \in C^{0,\alpha}_{loc}$.
We would expect that it would be possible to take $\beta=\alpha_0$, the H\"{o}lder exponent given by Caffarelli theory (depending on $\Lambda, \lambda$ and the geometry of the domains). On one side our result is worse since in the case $\alpha \approx \alpha_0 \ll 1$ we have $ \beta \approx \alpha_0^2 \ll \alpha_0$. On the other side our result is \emph{adaptive} to the regularity of $u_0$ itself, and the dependence on $\frac{\Lambda}{\lambda}$ and the geometry appears only at the level of constants.
\vspace{2mm}
\end{oss}

\noindent
\textbf{Prior results.} There has been a recent line of works where a uniform bounds (in $\varepsilon$) is obtained for the second derivative of entropic potentials but with stronger hypotesis on the marginals: the first result in this direction is \cite{fathi2020proof} where the entropic equivalent of the \emph{Caffarelli contraction theorem} is proved (see also \cite{chewi2023entropic} for a simpler proof), and then this result is improved in various ways in the more recent work \cite{gozlan2025global}.
Despite the results are stronger in terms of smoothness of the potentials, in both these papers there is always a second-order assumption on the marginals: in the first one strongly log-concavity is required while in the second one $\rho$-convexity of the logarithm is required.

A new variational method to prove the (classical) $C^{0,\alpha}_{loc}$-regularity result by Caffarelli was proposed by Goldman and Otto in \cite{goldman2019variational} and then generalized in \cite{goldman2021quantitative, otto2021variational}. Recently this approach was used in \cite{gvalani2025} for the entropic optimal transport problem: the authors prove a H\"{o}lder estimate for the entropic map however it holds only at \emph{large} scales (with respect to $\varepsilon$).\vspace{2mm}


\noindent
\textbf{Convergence to classical Optimal Transport.} When $\varepsilon\to0^+$, under the settings of \eqref{density:1}-\eqref{density:2}, the  Boltzmann-Shannon-Entropy regularized Optimal Transport in equation \eqref{EOTP:1}  $\Gamma-$converges to the $2-$Wasserstein distance between the probability measures $\mu$ and $\nu$~\cite{CarDuvPeySch17,LeoGamma},
\begin{align}
    \label{intro:W2}
    W_2^2(\mu,\nu) := \inf\limits_{\pi \in \Pi(\mu,\nu)}\I_{\Omega_{1}\times\Omega_{2}} \frac{1}{2} |x-y|^2\,d\pi(x,y).
\end{align}
Moreover, the pair of Schr\"odinger potentials $(\varphi_{\varepsilon},\psi_{\varepsilon})$ in \eqref{dens:phi:1} converges to the pair of functions $(\varphi_0,\psi_0)$ so-called Kantorovich potentials  under a suitable normalization that solve the dual problem~\cite{NW1,Fig17}
\begin{equation}\label{OTP:2}
    W_2^2(\mu,\nu)=
     \sup
    \left\{ \I_{\Omega_{1}}\varphi\,d\mu + \I_{\Omega_{2}}\psi\,d\nu :
\varphi(x) + \psi(y) \leqslant \frac{1}{2}|x-y|^2, \, \forall x\in\Omega_{1},y\in\Omega_{2} \right\}.
\end{equation}

In the rest of the paper, we denote

\begin{equation}
    \label{Kant_poten}
    u_{0}(x):= \frac{|x|^2}{2}-\varphi_0(x)\quad \text{and}\quad
    v_0(y):= \frac{|y|^2}{2}-\psi_{0}(y).\vspace{5mm}
\end{equation}
Notice that of course a uniform estimate for $u_{\varepsilon}$ would imply the same regularity result for $u_0$, so it is natural to consider as hypotesis the ones in the classical case \eqref{density:1}, \eqref{density:2}.

\vspace{10pt}

\noindent
\textbf{Main idea of the proof of Theorem \ref{thm:main}:} The main ingredients of the proof are 
\begin{itemize}
    \item[(i)] Caffarelli's $C_{loc}^{0,\alpha}$-regularity of the optimal map (or $C_{loc}^{1,\alpha}$ regularity for Kantorovich potentials) for $2$-Wasserstein distance~\eqref{intro:W2}, see Theorem \ref{thm:Caf} below;
    \item[(ii)] A \textit{local $p$-detachment} result (Lemma \ref{lem:Detach_local});
    \item[(iii)] The first-order expansion of the Boltzmann-Shannon
regularized Optimal Transport problem~\eqref{EOTP:1} around the limit $\varepsilon\to 0^+$ by Carlier, Pegon and Tamanini~\cite{CPT23} (see also \cite{CarDuvPeySch17}), i.e. 
\[
        0 \leqslant \text{W}^2_{\varepsilon}(\mu,\nu)-\text{W}^2_2(\mu,\nu) 
        \leqslant
        \frac{n}{2}\varepsilon\log\left( \frac{1}{\varepsilon} \right) + O(\varepsilon).
\]

\item[(iv)] The points (ii) and (iii) together will imply an $L^p$-closedness result between $\nabla u_{\varepsilon}$ and $ \nabla u_0$. Therefore we show $u_{\varepsilon}$ and $u_0$ are close in $L^{\infty}$, and then using the representations of $\nabla u_{\varepsilon}$ and $\nabla^2 u_{\varepsilon}$ (see Subsection \ref{subs:repres}) we will derive uniform bounds for them to conclude the main result.
\end{itemize}

\vspace{2mm}

\noindent
\textbf{Organization of the paper:} In the next section, we discuss the background of Optimal Transport and entropic regularized Optimal Transport fitting into the settings we consider in this paper. Section \ref{sec:pre} contains the technical results to be used in the proof of local estimates of potentials. Section \ref{sec:loc_est} is devoted to showing local estimates for Schr\"odinger and Kantorovich potentials. In the last section, we provide the proof of Theorem \ref{thm:main}.


\section{Optimal Transport Theory for the quadratic cost}

We first recall that given two probability measures $\mu\in \mathcal{P}(\R^{n})$, $\nu\in \mathcal{P}(\R^{n})$ the Monge-Kantorovich problem with quadratic cost function is concerned with minimization problem
\eqref{intro:W2}. The existence of a minimizer $\pi_{0}$ (an optimal transport plan between measures $\mu$ and $\nu$) for problem \eqref{intro:W2} is a classical result, see for instance \cite{Vil2}. Moreover,  there exists a minimizer $\pi_{0} \in \Pi(\mu,\nu)$ if and only if  $\pi_{0}$ is concentrated on the graph of the subdifferential of a lower-semicontinuous and convex function $u_{0}$, see \cite{Bre87, Bre91, Mcc95}:
\begin{align*}
    \text{spt}(\pi_{0}) \subset \text{Graph}(\partial u_{0})
    &:= \left\{ (x,y)\in \R^{n}\times \R^{n} : y\in \partial u_{0}(x) \right\}
    \notag\\&
    =
    \left\{ (x,y)\in \R^{n}\times \R^{n} : u_{0}(x) + u_{0}^{*}(y)  = \inner{x}{y} \right\},
\end{align*}
where $u_{0}^{*}$ denotes Legendre transform given by 
\begin{equation*}
    u_{0}^{*}(y):= \sup\limits_{z\in\R^{n}}\left(\inner{y}{z}-u_{0}(z) \right), 
\end{equation*}
and the subdifferential of $u_{0}$ at $x$ is given by 
    \begin{align*}
        \partial u_{0}(x) &:= \left\{ p\in \R^{n} : u_{0}(z)\geqslant u_{0}(x) + \inner{p}{z-x} \text{ for all } z\in\R^{n}\right\}.
    \end{align*}
Moreover, under the assumption that $\mu$ and $\nu$ are absolutely continuous with respect to the Lebesgue measure, there is a pair of functions $(\varphi_{0},\psi_{0})$ maximizing the dual problem \eqref{OTP:2} that are differentiable almost everywhere on the support of $\mu$ and $\nu$, respectively, and they satisfy the following equations
\begin{equation*}
    \nabla u_{0}(x) = x - \nabla\varphi_{0}(x)\quad \text{ for }\mu \text{-a.e. } x\in\R^d,
\end{equation*}
\begin{equation*}
    \nabla u_{0}^{*}(y) = y - \nabla \psi_{0}(y)\quad \text{ for }\nu \text{-a.e. } x\in\R^d. 
\end{equation*}
 Here $(\varphi_{0}, \psi_{0})$ is a pair of optimal Kantorovich potentials, see for instance \cite{Bre87, Bre91}.

\noindent

\subsection{$C^{1,\alpha}$-regularity in classical Optimal Transport Theory}

 Under the assumptions \eqref{density:1}-\eqref{density:2}, it is known that there is a unique optimal transport plan $\pi_{0}:= (\text{Id},T)_{\#}\mu$, where $T = \nabla u_{0}$ is the so-called Brenier's optimal transport map from $\mu$ to $\nu$. As a consequence of marginal conditions of the optimal transport $\pi_{0}$ for \eqref{intro:W2} and change of variables, Kantorovich potential $u_{0}$ satisfies the Monge-Ampere equation 
\begin{align}
    \label{eq:MA}
    g(\nabla u_{0}(x))\det\left(\nabla^2u_{0}(x) \right) = f(x)\quad\text{in}\quad \Omega_1
\end{align}
in the Alexandrov sense with the natural boundary condition  
\begin{align*}
    \nabla u_{0}(\Omega_{1}) = \Omega_{2},
\end{align*}
see for instance \cite{Fig17, Gut16, Vil1} and references therein. Moreover, $u_{0}$ is the strictly convex function in $\Omega_{1}$, see \cite{Ca3, De13}, and $u_{0}$ is of class $C_{\loc}^{1,\alpha_0}$ with a universal exponent $\alpha_0 \in (0,1)$ depending only on $n$, $\lambda$ and $\Lambda$ by Caffarelli's pioneering result \cite{Ca2, Ca3}: 

\begin{teo}
    \label{thm:Caf}
    Let $u_0 : \Omega_1 \rightarrow \R$ be a strictly convex solution of \eqref{eq:MA} in the Alexandrov sense under the assumptions \eqref{density:1}-\eqref{density:2}. Then $u_{0}\in C^{1,\alpha_0}_{\loc}(\Omega_1)$ for a universal exponent $\alpha_0 \equiv \alpha_0 (n,\lambda,\Lambda) \in (0,1)$. Moreover, for every subset $\Omega_{1}'\ssubset \Omega_1$ there exists a constant $M_0$ depending on $\Omega_{1}'$ and the modulus of convexity of $u_0$ such that 
    \begin{align*}
        \sup\limits_{x\neq y \in \Omega_{1}'} \frac{|\nabla u_{0}(x)-\nabla u_{0}(y)|}{|x-y|^{\alpha_0}} \leqslant M_0.
    \end{align*}
\end{teo}

\subsection{Boltzmann-Shannon entropy
regularized Optimal Transport}

It is the fact that problem \eqref{EOTP:1} with a general bounded cost function $c : \R^{n}\times \R^{n} \rightarrow \R_{+}$ instead of the quadratic cost function 
\begin{align}
    \label{Gen:1}
    \text{OT}_{\varepsilon}(\mu,\nu):= \inf\limits_{\pi \in \Pi(\mu,\nu)}\I_{\R^{n}\times \R^{n}} c(x,y)\,d\pi(x,y) + 
    \varepsilon H(\pi | \mu\otimes \nu) 
\end{align}
admits a unique minimizer   $\pi_{\varepsilon}\in \Pi(\mu,\nu)$ as a consequence of the direct method in the calculus of variations and strict convexity of entropy. Moreover, $\pi_{\varepsilon}\in \Pi(\mu,\nu)$ is the minimizer of problem \eqref{Gen:1} under assumptions that probability measures $\mu$ and $\nu$ have positive density functions if and only if
\begin{equation*}
    \pi_{\varepsilon} = a_{\varepsilon}(x)b_{\varepsilon}(y)\exp\left(-\frac{c(x,y)}{\varepsilon}\right)\mu\otimes \nu,
\end{equation*}
see \cite{BorLewNus94, leonard2014}, where  the functions $a_{\varepsilon} : \R^{n}\rightarrow (0,\infty)$ and $  b_{\varepsilon} : \R^{n}\rightarrow (0,\infty)$ satisfy properties that $\ln a_{\varepsilon} \in L_{\mu}^{1}(\R^{n})$ and $\ln b_{\varepsilon}\in L_{\nu}^{1}(\R^{n})$, and solve the Schr\"odinger system of equations
\begin{equation*}
    \begin{cases}
        \displaystyle a_{\varepsilon}(x)\I_{\R^{n}}b_{\varepsilon}(y)\exp\left( -\frac{c(x,y)}{\varepsilon}\right)\,d\nu(y) = 1\quad \forall x\in \R^{n} \\
        \displaystyle b_{\varepsilon}(y)\I_{\R^{n}}a_{\varepsilon}(x)\exp\left( -\frac{c(x,y)}{\varepsilon}\right)\,d\mu(x) = 1 \quad \forall y\in \R^{n}.
    \end{cases}
\end{equation*}

There are many natural questions regarding the behaviour of regularized entropic optimal transport problem \eqref{Gen:1} such as the convergence of potentials, optimal plans, and optimal values with suitable quantitative convergence rates  when $\varepsilon\rightarrow 0^{+}$. This is an attracting topic for intensive research nowadays, see for instance \cite{CarDuvPeySch17,DG20,LeoGamma,Mik04}. In particular, we recall a recent interesting paper investigating the convergence of optimal values with a suitable convergence rate ~\cite[Theorem 1.1]{CPT23} which will be employed in Section \ref{sec:loc_est} below.

\begin{teo}[\cite{CPT23}]
    \label{thm:CPT}
    Let $\Omega_{X}, \Omega_{Y}\subset \R^{n}$ be convex open sets and $\mu, \nu \in \mathcal{P}(\R^{n})$ be absolutely continuous probability measures compactly supported on $\Omega_{X}$, $\Omega_{Y}$, respectively, with bounded densities. Suppose cost function $c : \Omega_{X}\times \Omega_{Y}\rightarrow \R_{+}$ is of class $C^{2}$ and infinitesimally twisted. Then
    \begin{equation*}
        \mathrm{OT}_{\varepsilon}(\mu,\nu) = \mathrm{OT}(\mu,\nu) + \frac{n}{2}\varepsilon\log\left( \frac{1}{\varepsilon} \right) + O(\varepsilon),
    \end{equation*}
    where 
    \begin{equation*}
        \mathrm{OT}(\mu,\nu) = \inf\limits_{\pi \in \Pi(\mu,\nu)}\I_{\R^{n}\times \R^{n}} c(x,y)\,d\pi(x,y).
    \end{equation*}
\end{teo}

\section{Preliminaries}\label{sec:pre}

We shall always denote by $c$ a generic positive constant, possible varying line to line, having dependencies on parameters using brackets, that is, for example, $c\equiv c(n,\lambda,\Lambda)$ means that $c$ is a positive constant depending only on $n,\lambda$ and $\Lambda$. We start this section with a technical propostion to be used later.

\begin{prop}
    \label{prop:convex}
    Let $\Omega\subset \R^{n}$ be a bounded open convex set. There exists a constant $c\equiv c(n,\Omega)$ such that 
    \begin{align*}
        \I_{\Omega\cap B_{r}(z)}\exp\left( -\frac{|x-z|}{r} \right)\,dx 
        \geqslant
        c|B_{r}(z)|
    \end{align*}
    holds whenever $z\in\overline{\Omega}$ and $0<r<1$.
\end{prop}
\begin{proof}
    First, we show that there exists a constant $c\equiv c(n,\Omega)$ such that 
    \begin{align}
        \label{convex:1}
        \frac{|\Omega\cap B_{r}(z)|}{|B_{r}(z)|} \geqslant c
    \end{align}
    for every ball $B_{r}(z)$ centered at $z\in \overline{\Omega}$ and of radius $0<r<1$. Let $B_{2R}(z_0)\subset \Omega$ be one of the largest balls contained in $\Omega$.  Let us consider a ball $B_{r}(z)$ centered at $z\in \overline{\Omega}$ and of radius $0<r<1$. Then we denote $Z_{R}$ a convex hull of $B_{R}(z_0)$ and $z$. We consider two cases depending on the position of center $z$.  

    \textbf{Case 1. $z\in \overline{\Omega}\setminus \overline{B_{R}}(z_0)$.} Let $B_{r_{z}}(x_{z})$ be the largest ball contained in $Z_{R}\cap B_{r}(z)$. $B_{r_z}(x_z)$ will be the image of $B_R(z_0)$ through an homotety centered at $z$ with ratio $\frac{r}{R+|z-z_0|}$;in particular we will have $r_z=\frac{ Rr}{R+|z-z_0|}$. In turn we get

     \begin{align*}
        \frac{|\Omega\cap B_{r}(z)|}{|B_{r}(z)|}
        \geqslant
        \frac{|B_{r_{z}}(x_{z})|}{|B_{r}(z)|}
        =
        \left( \frac{R}{R + |z-z_0|} \right)^{n}
        \geqslant 
         \left( \frac{R}{R + \diam(\Omega)} \right)^{n}
    \end{align*}

   \textbf{Case 2. $z\in \overline{B_{R}}(z_0)$. } For $0<r < R$, we have that $B_{r}(z)\subset B_{2R}(z_0)\subset \Omega$ and
 \begin{align*}
     \frac{|\Omega\cap B_{r}(z)|}{|B_{r}(z)|}
     \geqslant
     \frac{|B_{2R}(z_0)\cap B_{r}(z)|}{|B_{r}(z)|} \geqslant 1;
 \end{align*}
 for $1> r\geqslant R$, we find 
  \begin{align*}
     \frac{|\Omega\cap B_{r}(z)|}{|B_{r}(z)|}
     \geqslant
     \frac{|B_{2R}(z_0)\cap B_{r}(z)|}{|B_{r}(z)|} 
     \geqslant
     \frac{|B_{R}(0)|}{|B_{1}(0)|}
     \geqslant R^{n}.
 \end{align*}
Taking into account two cases we have discussed above, we arrive at the validity of \eqref{convex:1}. Finally, using inequality \eqref{convex:1}, we have 
\begin{align*}
    \I_{\Omega\cap B_{r}(z)} \exp\left(-\frac{|x-z|}{r}\right)\,dx
    \geqslant
    e^{-1}|\Omega\cap B_{r}(z)| \geqslant c |B_{r}(z)|
\end{align*}
for some constant $c\equiv c(n,\Omega)$. 
\end{proof}

\subsection{Local $p$-detachment}

The following lemma establishes a lower bound for $u_{0}(x) + v_{0}(y) - \langle x, y \rangle$, where $u_{0}$ and $v_{0}$ are Kantorovich potentials defined in \eqref{Kant_poten}. Specifically, it shows that the separation of $u_{0}(x) + v_{0}(y)$ from the linear function $\langle x, y \rangle$ can be controlled in terms of a power of the distance between $y$ and the gradient $\nabla u_{0}(x)$.

\begin{deff}[$p$-detachment] Let $\Omega_X, \Omega_Y \subseteq \R^n$ and $f: \Omega_X \to \mathbb{R}$, $g:\Omega_Y \to \mathbb{R}$ be two convex functions such that $f(x)+g(y) \geqslant \langle x,y\rangle$ for every $(x,y) \in \Omega_X \times \Omega_Y$. We say that $(f,g)$ satisfy a $p$-detachment on $\Omega_X \times \Omega_Y$ if there exists $C>0$ such that 
$$f(x)+g(y) - \langle x,y \rangle \geqslant C \| y- \nabla f (x)\|^{p} \qquad \forall (x,y) \in \Omega_X \times \Omega_Y$$
\end{deff}
In the following fundamental lemma we show that if $\nabla f$ is $\alpha$-H\"{o}lder then $(f,g)$ locally satisfies a $p$-detachment property with $p= \frac{ 1+\alpha}{\alpha}$. This fact in the case $\alpha=1$ is already known in the Optimal Transport literature, that is if $\nabla f$ is Lispchitz then $(f,g)$ has quadratic detachment (see \cite{ambrosio13,Vil2}): we still include the proof to highlight the differences with respect to the H\"older case.

Notice that a converse result with similar assumptions has been noted in Remark 7.10 in \cite{caffarelli2010free}, where the $p$-uniform convexity of $f^*$ is shown to imply $\alpha$-H\"{o}lder regularity of $f$ (notice that $p$-detachment for $(f,f^*)$ implies that $f^*$ is $p$-uniform convex).

A similar idea is employed in \cite{conforti2024weak, conforti2023quantitative, gozlan2025global} but they need a more global result while we are more interested in the local version: moreover we state the coercivity (or strict convexity) as a detachment property, since the Young inequality residual is what appears in the duality gap.

\begin{lemma} (Local quadratic and $p$-detachment)
\label{lem:Detach_local}
Let $\Omega_X, \Omega_Y \subseteq \mathbb{R}^{n}$ be two bounded convex open sets. Let $u:\Omega_X \to \mathbb{R}$ be a convex function such that $\nabla u (\Omega_X)=\Omega_Y$. Let $v:\Omega_Y \to \mathbb{R}$ be such that
\begin{equation}\label{eqn:uvOmegalocal} u(x) + v(y) \geqslant \langle x, y \rangle \qquad \forall (x,y) \in \Omega_X \times \Omega_Y.\end{equation}
\begin{itemize}
\item[(i)] Suppose $u \in C^{1,1}(\Omega_X)$; let $\lambda= Lip(\nabla u, \Omega_X)$ and $L=\frac 1{2 \lambda}$. Then we have
\begin{equation}
\label{eqn:local2detachement} 
u(x) + v(y) -\langle x, y \rangle \geqslant  L \| y- \nabla u (x) \|^2 \qquad \forall x \in \Omega_X , \forall y \in \Omega_Y.
\end{equation}
    \item[(ii)] Suppose $u \in C^{1, \alpha}_{loc}(\Omega_X)$ and let $p=\frac {1+\alpha}{\alpha}$ . Then for every compact $K \ssubset \Omega_X $ there exists $L\equiv L(\alpha,K)>0$ such that
\begin{equation}
\label{eqn:localpdetachement} 
u(x) + v(y) -\langle x, y \rangle \geqslant  L \| y- \nabla u (x) \|^p \qquad \forall x \in K , y \in \Omega_Y.
\end{equation}

\end{itemize}
\end{lemma}

\begin{proof} 
Let us consider the convex conjugate of $u$ restricted to $\Omega_Y$, $\tilde{v}:= (u)^*$. By definition we have $u(x)+\tilde{v}(y) \geqslant \langle x, y \rangle $ and moreover
$$\tilde{v}(y) = \sup_{ x \in \Omega_X } \left\{ \langle x, y \rangle - u(x) \right\}; $$
using \eqref{eqn:uvOmegalocal} we obtain $\tilde{v}(y) \leqslant v(y)$ on $\Omega_Y$. In particular it is sufficient to prove \eqref{eqn:localpdetachement} and \eqref{eqn:local2detachement} for $u$ and $\tilde{v}$.
\begin{itemize}
\item[(i)] Since $\nabla u$ is $\lambda$-Lipschitz, it is differentiable almost everywhere and moreover we have $D^2u \leq \lambda Id$. $\nabla \tilde{v}$ is the inverse function of $\nabla u$: differentiating the relation $\nabla \tilde{v} ( \nabla u (x))=x$  we get $D^2 \tilde{v}  ( \nabla u (x) ) \cdot D^2 u =Id$. Since this is true for every $x \in \Omega_X$ and $\nabla u(\Omega_X)=\Omega_Y$ we obtain $D^2 \tilde{v}(y) \geq \frac 1{\lambda} Id$ for almost every $y \in \Omega_Y$. Thanks to the fact that $\Omega_Y$ is convex we have that $\tilde{v}$ is a $\frac 1{\lambda}$-strongly convex function and so 
$$ \tilde{v}(y) \geqslant \tilde{v}(\bar{y}) + \langle \nabla \tilde{v} (\bar(y)), y- \bar{y}\rangle + \frac 1{2\lambda} \| y- \bar{y}\|^2.$$
Letting $\bar{y}:=\nabla u (x)$, we have $\nabla \tilde{v} (\bar{y})=x$ and using the equality in the Legendre duality $\tilde{v}(\bar{y}) + u (x) = \langle \bar{y} , x\rangle$, we find that \eqref{eqn:local2detachement} holds true for $v= \tilde{v}$.

\item[(ii)] Let us consider a compact set $K_1$ such that $K \ssubset K_1 \ssubset \Omega_X$ and denote $\lambda=[\nabla u]_{\alpha;K_1}$ which is finite thanks to the local H\"{o}lder regularity of $u$. Let $x_0 \in K$ such that $u$ is differentiable at $x_0$. Thanks to the convexity of $u$ and the H\"{o}lder condition on its gradient, for every $x \in K_1$
 \[
 u(x) \leqslant u(x_0) + \langle \nabla u (x) , x-x_0 \rangle \leqslant  u(x_0)+ \langle \nabla u (x_0) , x-x_0 \rangle + \lambda \| x-x_0\|^{1+\alpha}.
 \]
Using this estimate in the definition of $\tilde{v}$ we obtain
\begin{align*}
    \tilde{v}(y) &= \sup_{x \in \Omega_X} \{ \langle x,y\rangle - u(x)  \}  \\
    &\geqslant \sup_{x \in K_1} \{ \langle x,y\rangle - u(x)  \}  \\
    & \geqslant \sup_{x \in K_1} \{ \langle x,y\rangle - u(x_0) - \langle \nabla u (x_0) , x-x_0 \rangle - \lambda \| x-x_0\|^{1+\alpha}  \} \\
    & = \sup_{x \in K_1} \{ \langle x_0,y\rangle - u(x_0)  + \langle y-\nabla u (x_0) , x-x_0 \rangle - \lambda \| x-x_0\|^{1+\alpha}  \} \\
    & = \langle x_0,y\rangle - u(x_0) +\sup_{x \in K_1} \{ \langle y-\nabla u (x_0) , x-x_0 \rangle - \lambda \| x-x_0\|^{1+\alpha}  \} \\
     & \geqslant \langle x_0,y\rangle - u(x_0) +  C(\lambda,\dist(x_0, \partial K_1))\| y- \nabla u (x_0)\|^p.
\end{align*}
We can conclude thanks to the fact that $\dist(x_0, \partial K_1)$ is bounded from below for $x_0 \in K$ since $K \ssubset K_1$. In fact, we simply select 
$$x = x_{0} + r\|y-\nabla u(x_0)\|^{\frac{1}{\alpha}-1}(y-\nabla u(x_0))$$
in the previous display for positive constant $r$ given by
\begin{align*}
    r = \min\left\{1/2\lambda^{-\frac{1}{\alpha}}, 1/2\dist(K,\partial K_{1})(\diam(\Omega_{Y}))^{-\frac{1}{\alpha}}\right\}.
\end{align*}
\end{itemize}
\end{proof}

\begin{oss} In the optimization literature it is folklore that H\"older continuity of the gradient of a convex function defined on the whole space is actually equivalent to a global detachment property for the convex conjugate on the whole space. We provide a self-contaited proof of this fact.

\textbf{(i) $f \in C^{1,\alpha}$ implies $(f,f^*)$ has $ \frac {1+\alpha}{\alpha}$-detachment}. Suppose $f \in C^1(\R^n)$ and $\nabla f$ is $\alpha$-H\"{o}lder continuous with constant $C= \frac{ \lambda^{1+\alpha}}{1+\alpha}$ and denote $p = \frac{ 1+\alpha}{\alpha}$. Then
\begin{align*}
    f^*(y) &= \sup_{x \in \R^n} \{ \langle x,y\rangle - f(x)  \}  \\
    & \geqslant \sup_{x \in \R^n} \{ \langle x,y\rangle - f(x_0) - \langle \nabla f(x_0) , x-x_0 \rangle - C \| x-x_0\|^{1+\alpha}  \} \\
    & = \sup_{x \in K_1} \{ \langle x_0,y\rangle -f(x_0)  + \langle y-\nabla f(x_0) , x-x_0 \rangle - C \| x-x_0\|^{1+\alpha}  \} \\
    & = \langle x_0,y\rangle - u(x_0) +\sup_{x \in \R^n} \{ \langle y-\nabla f (x_0) , x-x_0 \rangle - \frac{\lambda^{1+\alpha}}{1+\alpha} \| x-x_0\|^{1+\alpha}  \} \\
     & = \langle x_0,y\rangle - f(x_0) + \frac{ 1}{p\lambda^{p}}\| y- \nabla f (x_0)\|^p,
\end{align*}

\textbf{(ii) $(f,f^*)$ has $\frac{1+\alpha}{\alpha}$-detachment implies $f \in C^{1,\alpha}$.} Let $p=\frac{1+\alpha}{\alpha}$; suppose that
$$ f^*(y)+f(x) - \langle x,y\rangle  \geqslant \frac 1{p \lambda^p} \| y-q\|^p \qquad  \forall y \in \R^n, \forall q \in \partial f (x).$$
Considering $y=q' \in \partial f (x')$, using the Young identity $f^*(q')+f(x')= \langle q',x'\rangle$ and the subdifferential inequality $f(x') \geqslant f(x) + \langle p, x'-x\rangle$ we get
$$  \langle x-x', q'-q \rangle  \geqslant \frac 1{p \lambda^p} \| q'-q\|^p.$$
Notice that this inequality is precisely $p$-convexity for $\nabla f^*$. By Cauchy-Schwarz now we conclude $\|q'-q\| \leqslant p^{\alpha}\lambda^{1+\alpha} \|x'-x\|^{\alpha}$, so $f \in C^{1, \alpha}(\R^n)$, with the constant of H\"{o}lder continuity being $\bar{C}= p^{\alpha} \lambda^{1+\alpha}$.

Even though this global analysis holds true (the constants can be also improved, at least in the Lipschitz case), a complete \emph{local} theory for this equivalence seems to lack, even in the Lipschitz case. It has been conjectured that given $u : \Omega_X \to \mathbb{R}$ a convex function then $\nabla u$ is $L$-Lipschitz if and only if 
$$u(x)+v(y)-\langle x,y \rangle \geqslant \frac 1{2L} \| y- \nabla u(x)\|^2,$$
where $v$ is the convex conjugate of $u$. However, the Lipschitz assumption alone on $\nabla u$ is not enough to guarantee this detachment: see a counterexample in \cite{Drori20}. This is not in contradiction with Lemma~\ref{lem:Detach_local} as we assume also $\nabla u ( \Omega_X)$ to be convex while this is not true in the counter-example.

It would be interesting to know whether a uniform H\"{o}lder condition on $\nabla u$ and convexity of $\nabla u (\Omega_X)$ could imply a $p$-detachment with a uniform constant.
\end{oss}

\subsection{Relations and representation formulas}
\label{subs:repres}

We now want to exploit the characterization of $u_{\varepsilon}$ and $v_{\varepsilon}$ given in \eqref{schrod:2} to deduce representation formulas for their gradient and Hessian. In fact taking the derivative in $x$ in the first equation of \eqref{schrod:2} and in $y$ in the second equation of \eqref{schrod:2}, respectively, we get
\begin{equation}
\label{eq:1st:1} 
\nabla u_{\varepsilon}(x)  = \I_{\Omega_2} y \cdot  \exp\left(\frac{\inner{x}{y}-v_{\varepsilon}(y)-u_{\varepsilon}(x)}{\varepsilon}\right)\,d\nu(y).
\end{equation}
and 
\begin{equation}
\label{eq:1st:2} 
\nabla v_{\varepsilon}(y)  = \I_{\Omega_1} y \cdot  \exp\left(\frac{\inner{x}{y}-v_{\varepsilon}(y)-u_{\varepsilon}(x)}{\varepsilon}\right)\,d\mu(x).
\end{equation}

Taking another derivative we get
\begin{align} \nabla^2 u_{\varepsilon}(x) &=\frac 1{\varepsilon} \I_{\Omega_2} y  \otimes (y- \nabla u_{\varepsilon}(x)) \exp \left(\frac{\inner{x}{y}-v_{\varepsilon}(y)-u_{\varepsilon}(x)}{\varepsilon}\right)\,d\nu(y) \notag\\
&  \label{eq:2nd:1}  = \frac 1{\varepsilon} \I_{\Omega_2} (y- \nabla u_{\varepsilon}(x))  \otimes (y- \nabla u_{\varepsilon}(x)) \exp \left(\frac{\inner{x}{y}-v_{\varepsilon}(y)-u_{\varepsilon}(x)}{\varepsilon}\right)\,d\nu(y).
\end{align}
and 
\begin{align} \nabla^2 v_{\varepsilon}(y) &=\frac 1{\varepsilon} \I_{\Omega_1} x  \otimes (x- \nabla v_{\varepsilon}(y)) \exp \left(\frac{\inner{x}{y}-v_{\varepsilon}(y)-u_{\varepsilon}(x)}{\varepsilon}\right)\,d\mu(x) \notag\\
&  \label{eq:2nd:2}  = \frac 1{\varepsilon} \I_{\Omega_1} (x- \nabla v_{\varepsilon}(y))  \otimes (x- \nabla v_{\varepsilon}(y)) \exp \left(\frac{\inner{x}{y}-v_{\varepsilon}(y)-u_{\varepsilon}(x)}{\varepsilon}\right)\,d\mu(x).
\end{align}

Another ingredient  which will be useful in the next section is the relation between the non-optimality of a plan and the detachment. In particular given any plan $\pi \in \Pi(\mu, \nu)$ we have
\begin{align} \frac 12 \I_{\Omega_1 \times \Omega_2} \|x-y\|^2 \,d\pi - W_2^2(\mu, \nu) &=  \I_{\Omega_1 \times \Omega_2} \left(\frac 12 \|x-y\|^2 - \varphi_0(x) - \psi_0(y) \,\right) d \pi \notag\\
& \label{eqn:detachment} = \I_{\Omega_1 \times \Omega_2} (u_0(x)+v_0(y) - \langle x, y \rangle ) \, d \pi,
\end{align}
where we have used \eqref{OTP:2}-\eqref{Kant_poten}.







\section{Local estimates of potentials}\label{sec:loc_est}

Throughout this section whatever we prove below will be considered under the assumptions \eqref{density:1}-\eqref{density:2}.
First we discuss the local convergence of the gradient of Schr\"odinger potentials $\{u_{\varepsilon}\}_{\varepsilon>0}$ given in \eqref{schrod:2} to Kantorovich potential $u_{0}$ in \eqref{Kant_poten} with a proper rate of convergence in $L^{p}_{\mu}(\Omega_{1})$ for every $0<p<\infty$ as $\varepsilon\rightarrow 0^{+}$.

\begin{lemma}
    \label{lem:UE}
    Let $u_{\varepsilon} : \Omega_1\rightarrow \R$ and $v_{\varepsilon} : \Omega_2\rightarrow \R$ be Schr\"odinger potentials satisfying \eqref{schrod:2} and strictly convex functions $u_0 : \Omega_{1}\rightarrow \R$, $v_{0} : \Omega_{2}\rightarrow \R$ be Kantorovich potentials given in \eqref{Kant_poten}. Then there exists a constant $p_{0}\equiv p_{0}(n,\lambda,\Lambda)$ such for for every $p\geq p_{0}$ and every compact subset $\Omega_{1}'\Subset \Omega_{1}$ there is a universal constant $c \equiv c(n,\lambda,\Lambda,p,\Omega_{1}',\Omega_{1})$ such that 
    \begin{equation}
        \label{lem:UE:1}
        \I_{\Omega_{1}'} |\nabla u_{\varepsilon}(x)-\nabla u_{0}(x)|^{p}\,d\mu(x)
        \leqslant
        c\varepsilon\log\left( \frac{1}{\varepsilon} \right) + O(\varepsilon)
    \end{equation}
\end{lemma}

\begin{proof}
   Note that we can apply Lemma \ref{lem:Detach_local} to $u_{0}$ since $u_{0}\in C^{1,\alpha}_{\loc}(\Omega_{1})$ for an exponent $\alpha\equiv \alpha(n,\lambda,\Lambda)\in (0,1)$ by Theorem \ref{thm:Caf}. In turn, for $p =  \frac{1+\alpha}{\alpha}$, using the duality of optimal transport and applying Lemma \ref{lem:Detach_local} to $u_{0}$, we see 
\begin{align*}
    \I_{\Omega_{1}'} |\nabla u_{\varepsilon}(x)-\nabla u_{0}(x)|^{p}\,d\mu(x)
    &\leqslant
    \I_{\Omega_{1}} |\nabla u_{\varepsilon}(x)-\nabla u_{0}(x)|^{p}\chi_{\Omega_{1}'}(x)\,d\mu(x)
    \\&
    \leqslant
    \I_{\Omega_{1}\times \Omega_{2}} |\nabla u_{0}(x)-y|^{p}\chi_{\{ \Omega_{1}'\times\Omega_{2}\}}(x,y)\,d\pi_{\varepsilon}(x,y)
    \\&
    \leqslant
    \frac 1L \I_{\Omega_{1}\times \Omega_{2}} \left( u_{0}(x) + v_{0}(y)-\inner{x}{y} \right)\chi_{\{ \Omega_{1}'\times\Omega_{2}\}}(x,y)\,d\pi_{\varepsilon}(x,y)
    \\&
    \leqslant
    \frac 1L \I_{\Omega_{1}\times \Omega_{2}} \left( u_{0}(x) + v_{0}(y)-\inner{x}{y} \right)\,d\pi_{\varepsilon}(x,y)
    \\&
    \leqslant
    \frac 1L\left[ W_{\varepsilon}^{2}(\mu,\nu)-W_{2}^{2}(\mu,\nu)\right]
    \leqslant
    c\varepsilon\log\left(\frac{1}{\varepsilon}\right) + O(\varepsilon)
\end{align*}
for some constant $c\equiv c(n,\lambda,\Lambda,p,\Omega_{1}',\Omega_{1})$, where in the last display we have used the relation \eqref{eqn:detachment} and Theorem \ref{thm:CPT}. Therefore, we find the constant $p_{0} = \frac{1+\alpha}{\alpha}$ and the inequality \eqref{lem:UE:1} is still valid for any $p\geqslant p_0$ by selecting $\alpha$ close enough to $0$.
\end{proof}

In the following lemmas, we obtain local $L^{\infty}$ estimates for the difference of Schr\"odinger and Kantorovich potentials $u_{\varepsilon} - u_0$ with a suitable convergence rate of power of $\varepsilon$ and for their corresponding first and second derivatives as well.

\begin{lemma}[$L^{\infty}$-estimate]
    \label{lem:0E}
    Let $u_{\varepsilon} : \Omega_1\rightarrow \R$ and $v_{\varepsilon} : \Omega_2\rightarrow \R$ be Schr\"odinger potentials satisfying \eqref{schrod:2} and strictly convex functions $u_0 : \Omega_{1}\rightarrow \R$, $v_{0} : \Omega_{2}\rightarrow \R$ be Kantorovich potentials given in \eqref{Kant_poten}. Then there exists $a\equiv a(n,\lambda,\Lambda)\in (0,1)$ such that for every compact convex subset $\Omega_{1}'\ssubset \Omega_{1}$, the following estimate 
    \begin{align}
        \label{lem:0E:1}
        | u_{\varepsilon}(x)-u_0(x)-(u_{\varepsilon}(y)-u_{0}(y))|
        \leqslant A_{1}\varepsilon^{a}
    \end{align}
    holds for a constant $A_{1}$ depending only on $n,\lambda$, $\Lambda$, $\Omega_{1}'$ and $\Omega_{1}$, whenever $x,y\in \Omega_{1}'$. Moreover, for any compact convex subset $\Omega_{1}'\ssubset \Omega_{1}$, if the function $u_{\varepsilon}-u_{0}$ has a zero on $\Omega_{1}'$, then
    \begin{align}
        \label{lem:ABF:0}
        \norm{u_{\varepsilon}-u_0}_{L^{\infty}(\Omega_1')} \leqslant A_{1}\varepsilon^{a},
    \end{align}
    and
    \begin{align}
        \label{lem:ABF:3}
        v_{0}(y)-v_{\epsilon}(y) \leqslant A_{2}\varepsilon^{a}\quad\text{for all}\quad y\in\nabla u_{0}(\Omega_{1}')
    \end{align}
    for some constant $A_{2}\equiv A_{2}(n,\lambda,\Lambda,\Omega_{1}',\Omega_{1})$.
\end{lemma}

\begin{proof}
    Applying the triangle inequatity and Sobolev-Poincar\'e inequality, see for instance \cite[Theorem 7.10]{GT77}, for any points $x,y\in \Omega_{1}', $ we have 
\begin{align}
    \label{ABF:s1:3}
    | u_{\varepsilon}(x)-u_0(x)-(u_{\varepsilon}(y)-u_{0}(y))|
    &\leqslant
    |u_{\varepsilon}(x)-u_0(x)-(u_{\varepsilon}-u_{0})_{\Omega_{1}'}| 
    \notag\\&\quad 
    + |u_{\varepsilon}(y)-u_{0}(y)-(u_{\varepsilon}-u_{0})_{\Omega_{1}'}|
    \notag\\
    &\leqslant
    2\norm{u_{\varepsilon}-u_{0}-(u_{\varepsilon}-u_{0})_{\Omega_{1}'}}_{L^{\infty}(\Omega_{1}')} 
    \notag\\
    &\leqslant 
    c(n,p) |\Omega_{1}'|^{\frac{1}{n}-\frac{1}{p}}\norm{\nabla u_{\varepsilon}-\nabla u_{0}}_{L^{p}(\Omega_{1}')}
    \notag \\&
    \leqslant
    c(n,p)\lambda^{-\frac{1}{p}} |\Omega_{1}'|^{\frac{1}{n}-\frac{1}{p}} \left(\I_{\Omega_{1}'} |\nabla u_{\varepsilon}(x)-\nabla u_{0}(x)|^{p}\,d\mu(x) \right)^{\frac{1}{p}}
    \notag \\&
    \leqslant
    c(n,p) \lambda^{-\frac{1}{p}} |\Omega_{1}|^{\frac{1}{n}-\frac{1}{p}} 
    \left[ \frac{Ln}{2}\varepsilon\log\left(\frac{1}{\varepsilon}\right) + O(\varepsilon) \right]^{\frac{1}{p}},
\end{align}
where we have used the estimate \eqref{density:2} and Lemma \ref{lem:UE} for $p>n$. Note that such choice of $p$ is possible by choosing $\alpha$ close enough to zero. Since Schr\"odinger potentials $(\varphi_{\varepsilon}, \psi_{\varepsilon})$ are unique up to adding and subtracting a constant, the pair $(u_{\varepsilon}-d, v_{\varepsilon}+d)$ still satisfies the same system of equations \eqref{schrod:2}. Therefore, we can replace $u_{\varepsilon}$ with $u_{\varepsilon}-d$ for a constant $d$ such that the function $u_{\varepsilon}-u_0$ has a zero on $\Omega_{1}'$. This observation together with the estimate \eqref{ABF:s1:3} follows the estimates \eqref{lem:0E:1} and \eqref{lem:ABF:0} if $\varepsilon \ll 1$. It is clear that the estimates \eqref{lem:0E:1} and \eqref{lem:ABF:0} hold true when parameter $\varepsilon$ bounded below from zero. 

Now we focus on showing the estimate \eqref{lem:ABF:3}. Using the second equation of \eqref{schrod:2} and the local estimate \eqref{lem:ABF:0}, for every $y\in \Omega_{2}':= \nabla u_{0}(\Omega_{1}')$, we have 
\begin{align}
    \label{ABF:s2:1}
    \exp\left( \frac{v_{\varepsilon}(y)-v_{0}(y)}{\varepsilon}\right) 
    &=
    \I_{\Omega_{1}} \exp\left( \frac{\inner{x}{y}-u_{\varepsilon}(x)-v_{0}(y)}{\varepsilon} \right)\,d\mu(x)
    \notag\\&
    \geqslant
    \I_{\Omega_{1}'} \exp\left( \frac{\inner{x}{y}-u_{0}(x)-v_{0}(y)}{\varepsilon} \right) \exp\left( \frac{u_{0}(x)-u_{\varepsilon}(x)}{\varepsilon} \right)\,d\mu(x)
    \notag\\&
    \geqslant
    \lambda \exp\left( -A_{1} \varepsilon^{a-1} \right)\I_{\Omega_{1}'} \exp\left( \frac{-L|x-\nabla v_{0}(y)|}{\varepsilon} \right)\,d\mu(x),
\end{align}
where we have also used the property that the function $F(x):= u_0(x)+v_0(y) - \langle x, y \rangle $ is $L$-Lipschitz with $L=\diam(\Omega_2)$ (since $|\nabla F (x)|= |\nabla u_0(x)- y| \leqslant  \diam(\Omega_2)$) and moreover $F(\bar{x})=0$ for $\bar{x}= \nabla v_0(y)$ and so
\begin{align*}
    u_{0}(x) + v_{0}(y)-\inner{x}{y} = F(x)-F(\bar{x}) \leqslant L|x-\bar{x}| = L|x-\nabla v_{0}(y)|\quad \forall x\in \Omega_1, y\in \Omega_2.
\end{align*}
Recalling the basic facts \cite{ambrosio13, Bre91} that 
\begin{align*}
    \nabla v_{0}(\nabla u_{0}(x)) = x
    \quad
    \text{and}
    \quad
    \nabla u_{0}(\nabla v_{0}(y)) = y,
\end{align*}
we have $\nabla v_{0}(y)\in \Omega_{1}'$ and we continue to estimate \eqref{ABF:s2:1} as 
\begin{align}
    \label{ABF:s2:4}
    \exp\left( \frac{v_{\varepsilon}(y)-v_{0}(y)}{\varepsilon} \right)
    &\geqslant
    \lambda\exp\left( -A_{1}\varepsilon^{a-1} \right) \I_{\Omega_{1}'\cap B_{\varepsilon/L}(\nabla v_{0}(y))}\exp\left(\frac{-L|x-\nabla v_{0}(y)|}{\varepsilon}\,dx \right)
    \notag \\& 
    \geqslant
    c_{0}\varepsilon^{n}\exp(-A_{1}\varepsilon^{a-1})
\end{align}
for a constant $c_{0}\equiv c_{0}(\lambda,L,\Omega_{1}')$, where we have applied Proposition \ref{prop:convex} thanks to the convexity assumption of $\Omega_{1}'$. On other hand, recalling $a\in (0,1)$ and an elementary inequality that there exists a constant $A_{1}'\equiv A_{1}'(n,\lambda,\Lambda,L,\Omega_{1}')$ such that
\begin{align*}
    \frac{\varepsilon^{-n}}{c_{0}} \leqslant \exp\left( A_{1}'\varepsilon^{a-1} \right)\quad \forall \varepsilon\in (0,1).
\end{align*}
Therefore, this estimate together with the estimate \eqref{ABF:s2:4} implies the validity of \eqref{lem:ABF:3} when $\varepsilon\in (0,1)$. There is nothing to show for \eqref{lem:ABF:3} when $\varepsilon\geqslant 1$.
\end{proof}

The next result establishes a uniform convergence estimate for the gradients of the Schr\"odinger potentials in terms of the regularization parameter $\varepsilon$.

\begin{lemma}[Gradient estimate]
    \label{lem:1E}
    Let $u_{\varepsilon} : \Omega_1\rightarrow \R$ and $v_{\varepsilon} : \Omega_2\rightarrow \R$ be Schr\"odinger potentials satisfying \eqref{schrod:2} and strictly convex functions $u_0 : \Omega_{1}\rightarrow \R$, $v_{0} : \Omega_{2}\rightarrow \R$ be Kantorovich potentials given in \eqref{Kant_poten}. Then there exists a constant $b\equiv b(n,\lambda,\Lambda)>0$ such that for every convex compact subset $\Omega_{1}'\ssubset \Omega_{1}$ the following estimate 
    \begin{align}
        \label{lem:ABF:1}
        \norm{\nabla u_{\varepsilon}-\nabla u_0}_{L^{\infty}(\Omega_1')} \leqslant B\varepsilon^{b}
    \end{align}
    holds for a constant $B$ depending only on $n,\lambda$, $\Lambda$, $\Omega_{1}'$ and $\Omega_{1}$.
\end{lemma}

\begin{proof}
    Let us fix a convex compact open subset $\Omega_{1}'\ssubset \Omega_{1}$. Then we are able to assume that the function $u_{\varepsilon}$ is normalized on $\Omega_{1}'$ in the sense that the function $u_{\varepsilon}-u_{0}$ has a zero on $\Omega_{1}$ since the pair $(u_{\varepsilon}-d,v_{\varepsilon}+d)$ are Schr\"odinger potentials satisfying \eqref{schrod:2} for any constant $d$. In fact this normalization does not affect the gradients $Du_{\varepsilon}$ and $Dv_{\varepsilon}$. However, such normalization enables us to apply the estimates \eqref{lem:ABF:0} and \eqref{lem:ABF:3} of Lemma \ref{lem:0E}. The representation formula \eqref{eq:1st:2} and straightforward computations imply
\begin{align*}
    |\nabla v_{\varepsilon}(y)-\nabla v_{0}(y)| 
    &\leqslant
    \I_{\Omega_{1}} |x-\nabla v_{0}(y)|\exp\left( \frac{\inner{x}{y}-u_{\varepsilon}(x)-v_{\varepsilon}(y)}{\varepsilon} \right)\,d\mu(x)
    \notag\\&
    \leqslant
    \diam(\Omega_{1})\quad \forall y\in\Omega_{2}.
\end{align*}
Therefore, using the mean value theorem and the convexity of $\Omega_{2}$, for all $y_{1},y_{2}\in \Omega_{2}$, we find 
\begin{align}
    \label{ABF:s3:2}
    v_{0}(y_1)-v_{\varepsilon}(y_1)-(v_{0}(y_2)-v_{\varepsilon}(y_2))
    &\leqslant
    |\nabla v_{0}(y_{t})-\nabla v_{\varepsilon}(y_{t})||y_1-y_2|
    \notag\\&
    \leqslant
    \diam(\Omega_{1})|y_1-y_2|,
\end{align}
where $y_{t} = ty_1 + (1-t)y_{2}\in \Omega_{2}$ is a point for some $t\in [0,1]$. Then we denote $\Omega_{2}':= \nabla u_{0}(\Omega_{1}')$ here and in the rest of the proof.
It is clear \cite[Lemma A.22]{Fig17} that $\Omega_{2}'\ssubset \Omega_{2}$ ($\nabla u_{0}(\Omega_{1})$ is compact and $\nabla u_{0}(\Omega_{1}) = \Omega_{2}$), and we also fix a compact open subset $\Omega_{2}''\ssubset \Omega_{2}$ such that $\Omega_{2}'\ssubset \Omega_{2}'' \ssubset \Omega_{2}$ and 
\begin{align*}
    -\frac{L}{2}\left( \dist(\Omega_{2}',\partial \Omega_{2}'') + \dist(y,\partial \Omega_{2}'') \right)^{p} + 2\diam(\Omega_{1})\dist(y,\partial\Omega_{2}'') \leqslant 0 \quad \forall y\in \Omega_{2}\setminus \Omega_{2}'',
\end{align*}
in which the constant $L$ and $p$ are determined through the application of Lemma \ref{lem:Detach_local} (ii) for $(u_{0},v_{0})$ in the set $\Omega_{1}'\ssubset \Omega_{1}$. Note that the selection of set $\Omega_{2}''$ is possible and independent of the parameter $\varepsilon$.

Now we focus on showing the estimate \eqref{lem:ABF:1}. Using the representation formula \eqref{eq:1st:1}, for every $x\in\Omega_{1}'$, we have 
    \begin{align}
        \label{ABF:s3:4}
        |\nabla u_{\varepsilon}(x)-\nabla u_{0}(x)| 
        &\leqslant
        \I_{\Omega_2} |y-\nabla u_{0}(x)|\exp\left(\frac{\inner{x}{y}-u_{\varepsilon}(x)-v_{\varepsilon}(y)}{\varepsilon}\right)\,d\nu(y)
        \notag\\&
        =
        \I_{\Omega_2''} |y-\nabla u_{0}(x)|\exp\left(\frac{\inner{x}{y}-u_{\varepsilon}(x)-v_{\varepsilon}(y)}{\varepsilon}\right)\,d\nu(y)
        \notag\\&
        \quad +
        \I_{\Omega_{2}\setminus \Omega_{2}''} |y-\nabla u_{0}(x)|\exp\left(\frac{\inner{x}{y}-u_{\varepsilon}(x)-v_{\varepsilon}(y)}{\varepsilon}\right)\,d\nu(y)
        \notag\\&
        =: I_{\Omega_{2}''}(x) + J_{\Omega_{2}''}(x),
    \end{align}
where $\Omega_{2}' = \nabla u_{0}(\Omega_{1}')$ and $\Omega_{2}''\ssubset \Omega_{2}$ is the set with $\Omega_{2}'\ssubset \Omega_{2}''$ as we have fixed above.

Now we estimate the terms $I_{\Omega_{2}''}(x)$ and $J_{\Omega_{2}''}(x)$ in the last display. Note that the convex hull of $\nabla v_{0}(\Omega_{2}'')$ is compactly contained in $\Omega_{1}$ and contains $\Omega_{1}'$ where $u_{\varepsilon}-u_{0}$ has a zero on $\Omega_{1}'\subset \nabla v_{0}(\Omega_{2}'')$ which allows us to employ the estimates \eqref{lem:ABF:0} and \eqref{lem:ABF:3} of Lemma \ref{lem:0E}. 
Therefore, applying Lemma \ref{lem:Detach_local} to $(u_{0},v_{0})$ and using the estimates \eqref{lem:ABF:0} and \eqref{lem:ABF:3} of Lemma \ref{lem:0E}, we find
    \begin{align}   
        \label{ABF:s3:5}
        I_{\Omega_{2}''}(x) &= 
        \I_{\Omega_{2}''} |y-\nabla u_{0}(x)|\exp\left(\frac{\inner{x}{y}-u_{\varepsilon}(x)-v_{\varepsilon}(y)}{\varepsilon}\right)\chi_{\{|y-\nabla u_{0}(x)|\leqslant M_0\varepsilon^{m_0} \}}\,d\nu(y)
        \notag\\&
         \quad +
        \I_{\Omega_{2}''} |y-\nabla u_{0}(x)|\exp\left(\frac{\inner{x}{y}-u_{\varepsilon}(x)-v_{\varepsilon}(y)}{\varepsilon}\right)\chi_{\{|y-\nabla u_{0}(x)|>M_0\varepsilon^{m_0} \}}\,d\nu(y)        
        \notag\\&
        \leqslant
        M_0\varepsilon^{m_0} +
        \I_{\Omega_{2}''} |y-\nabla u_{0}(x)|\exp\left(\frac{\inner{x}{y}-u_{0}(x)-v_{0}(y)}{\varepsilon}\right)
        \notag\\&
        \qquad\qquad\qquad\times
        \exp\left(\frac{u_0(x)-u_{\varepsilon}(x)}{\varepsilon}\right) \exp\left(\frac{v_0(y)-v_{\varepsilon}(y)}{\varepsilon}\right)
        \chi_{\{|y-\nabla u_{0}(x)|>M_0\varepsilon^{m_0} \}}\,d\nu(y)
        \notag\\&
        \leqslant
         M_0\varepsilon^{m_0} +
        \I_{\Omega_{2}''} |y-\nabla u_{0}(x)|\exp\left( (A_{1}+A_{2})\varepsilon^{a-1}-L|y-\nabla u_{0}(x)|^{p}\varepsilon^{-1}\right)
         \notag\\&
        \qquad\qquad\qquad\times
        \chi_{\{|y-\nabla u_{0}(x)|>M_0\varepsilon^{m_0} \}}\,d\nu(y),
    \end{align}
where the constants $M_0$ and $m_0$ in the display above to be selected as follows:
\begin{align*}
    M_0 = \left( \frac{2(A_{1}+A_{2})}{L} \right)^{\frac{1}{p}}
    \quad\text{and}\quad
    m_0 = \frac{a}{p}.
\end{align*}
Note that the constants $A_{1}$ and $A_{2}$ are universal (independent of $\varepsilon$) and determined via Lemma \ref{lem:0E}. In turn, inserting the selection of $M_0$ and $m_0$ in the last into \eqref{ABF:s3:5},  for any $x\in \Omega_{1}'$, we obtain 
\begin{align*}
    I_{\Omega_{2}''}(x) \leqslant
    \left( \frac{2(A_{1}+A_{2})}{L} \right)^{\frac{1}{p}}\varepsilon^{\frac{a}{p}} + \I_{\Omega_{2}''} |y-\nabla u_0(x)|\exp\left(-\frac{L}{2}|y-\nabla u_0(x)|^{p}\varepsilon^{-1} \right)\,d\nu(y).
\end{align*}
It is straightforward to see that a function $g : [0,\infty)\rightarrow [0,\infty)$ given by 
\begin{align*}
    g(t) := t\exp \left(-L/2t^{p}\varepsilon^{-1} \right)
\end{align*}
attains its maximum value of $\left( \frac{2\varepsilon}{Lp} \right)^{\frac{1}{p}}\exp\left(  - \frac{1}{p} \right)$ at $t = \left(\frac{2\varepsilon}{Lp} \right)^{\frac{1}{p}}$. Therefore, we find
\begin{align}
    \label{ABF:s3:9}
    I_{\Omega_{2}''}(x)
    \leqslant
    \left( \frac{2(A_{1}+A_{2})}{L} \right)^{\frac{1}{p}}\varepsilon^{\frac{a}{p}} + 
    \left( \frac{2\varepsilon}{Lp} \right)^{\frac{1}{p}}\exp\left(- \frac{1}{p} \right)
    \quad\forall x\in \Omega_{1}'.
\end{align}

Now we turn our attention to estimating $J_{\Omega_{2}''}(x)$ in \eqref{ABF:s3:4}. For this, again applying Lemma \ref{lem:Detach_local} for $(u_0,v_0)$ in the set $\Omega_{1}'$ and using \eqref{lem:ABF:0}, for every $x\in \Omega_{1}'$, we have 
\begin{align}
    \label{ABF:s3:10}
    J_{\Omega_{2}''}(x) 
    &= 
    \I_{\Omega_{2}\setminus\Omega_{2}''} |y-\nabla u_{0}(x)|\exp\left(\frac{\inner{x}{y}-u_{0}(x)-v_{0}(y)}{\varepsilon}\right)
        \exp\left(\frac{u_0(x)-u_{\varepsilon}(x)}{\varepsilon}\right) 
    \notag\\&
    \qquad\times    
        \exp\left(\frac{v_0(y)-v_{\varepsilon}(y)}{\varepsilon}\right)\,d\nu(y)
    \notag\\&
    \leqslant
    \I_{\Omega_{2}\setminus\Omega_{2}''} |y-\nabla u_{0}(x)|\exp\left(-L|y-\nabla u_{0}(x)|^{p}\varepsilon^{-1} + A_{1}\varepsilon^{a-1}\right)
    \exp\left(\frac{v_0(y)-v_{\varepsilon}(y)}{\varepsilon}\right)\,d\nu(y)
\end{align}
for a constant $A_{1}\equiv A_{1}(n,\lambda,\Lambda,\Omega_{1}',\Omega_{1})$ determined via Lemma \ref{lem:0E}. At this moment, recalling \eqref{ABF:s3:2}, for every $y\in \Omega_{2}\setminus \Omega_{2}''$ there is a point $z_{y}\in \Omega_{2}''$ such that 
\begin{align*}
    v_{0}(y)-v_{\varepsilon}(y) \leqslant v_{0}(z_{y})-v_{\varepsilon}(z_{y}) + 2\diam(\Omega_{1})\dist(y,\partial\Omega_{2}'').
\end{align*}
Inserting this estimate in the last display into \eqref{ABF:s3:10} and applying \eqref{lem:ABF:3} for the convex closure of set $\Omega_{1}'' = \nabla v_{0} (\Omega_{2}'')$, we find 
\begin{align}
    \label{ABF:s3:12}
    J_{\Omega_{2}''}(x) 
    &\leqslant
    \I_{\Omega_{2}\setminus\Omega_{2}''} |y-\nabla u_{0}(x)|\exp\left(-L|y-\nabla u_{0}(x)|^{p}\varepsilon^{-1} + A_{1}\varepsilon^{a-1}\right)
    \exp\left(\frac{v_0(y)-v_{\varepsilon}(y)}{\varepsilon}\right)\,d\nu(y)
    \notag\\&
    \leqslant
    \I_{\Omega_{2}\setminus\Omega_{2}''} |y-\nabla u_{0}(x)|
     \notag\\&
    \times
    \exp\left(-L|y-\nabla u_{0}(x)|^{p}\varepsilon^{-1} + (A_{1}+A_{2})\varepsilon^{a-1} +2\diam(\Omega_{1})\dist(y,\partial\Omega_{2}'')\varepsilon^{-1}\right)
    \,d\nu(y)
\end{align}
Using the triangle inequality and \eqref{ABF:s3:2}, we observe that 
\begin{align*}
    &-\frac{L}{2}|y-\nabla u_{0}(x)|^{p} + 2\diam(\Omega_{1})\dist(y,\partial\Omega_{2}'')
    \notag\\&
    \leqslant
    -\frac{L}{2}\left( \dist(\Omega_{2}',\partial \Omega_{2}'') + \dist(y,\partial \Omega_{2}'') \right)^{p} + 2\diam(\Omega_{1})\dist(y,\partial\Omega_{2}'') \leqslant 0 
\end{align*}
holds for every $x\in \Omega_{1}'$ and $y\in \Omega_{2}\setminus \Omega_{2}''$ thanks to the choice of $\Omega_{2}''$. Therefore, plugging the last estimate in \eqref{ABF:s3:12}, we find 
\begin{align*}
    J_{\Omega_{2}''}(x) \leqslant
    \I_{\Omega_{2}\setminus \Omega_{2}''} |y-\nabla u_{0}(x)|\exp\left(-L/2|y-\nabla u_{0}(x)|^{p}\varepsilon^{-1} + (A_{1}+A_{2})\varepsilon^{a-1}\right)\,d\nu(y)\quad
    \forall x\in\Omega_{1}.
\end{align*}
Arguing similarly as for $I_{\Omega_{2}''}(x)$ in \eqref{ABF:s3:5}-\eqref{ABF:s3:9} $(x\in \Omega_{1}')$ together with the estimate \eqref{ABF:s3:9}, we arrive at the validity of the estimate \eqref{lem:ABF:1}.

\end{proof}

While the Kantorovich potentials are known to be smooth and strictly convex under suitable regularity conditions on the data, the Schr\"odinger potentials inherit their regularity from the entropic regularization and may exhibit increasingly singular features as $\varepsilon \to 0^{+}$. The following lemma provides a quantitative upper bound on the Hessian of $u_\varepsilon$, showing that the potentials can became increasingly steep as the regularization vanishes, but in a controlled way. Notice that thanks to the representation formula \eqref{eq:2nd:1} as long as $\Omega_Y$ is bounded we have $\| \nabla^2 u_{\varepsilon} \|_{\infty} \lesssim \frac 1{\varepsilon}$: we improve this estimate, crucially allowing for an exponent $m <1$.

\begin{lemma}[Hessian estimate]
    \label{lem:ABF}
    Let $u_{\varepsilon} : \Omega_1\rightarrow \R$ and $v_{\varepsilon} : \Omega_2\rightarrow \R$ be Schr\"odinger potentials satisfying \eqref{schrod:2} and strictly convex functions $u_0 : \Omega_{1}\rightarrow \R$, $v_{0} : \Omega_{2}\rightarrow \R$ be Kantorovich potentials given in \eqref{Kant_poten}. Then there exist a constant $m\equiv m(n,\lambda,\Lambda)\in (0,1)$ such that  for every convex compact subset $\Omega_{1}'\ssubset \Omega_{1}$, the following estimates 
    \begin{align}
        \label{lem:ABF:2}
        \norm{\nabla^2u_{\varepsilon}}_{L^{\infty}(\Omega_1')} \leqslant \frac{M}{\varepsilon^{m}}
    \end{align}
    hold for a constant $M$ depending only on $n,\lambda$, $\Lambda$, $\Omega_{1}'$ and $\Omega_{1}$. 
\end{lemma}
\begin{proof}
 The representation formula \eqref{eq:2nd:1} reads as 
\begin{align*}
    \nabla^2u_{\varepsilon}(x) = \frac{1}{\varepsilon}\I_{\Omega_{2}} (y-\nabla u_{\varepsilon}(x))\otimes (y-\nabla u_{\varepsilon}(x))\exp\left( \frac{\inner{x}{y}-u_{\varepsilon}(x)-v_{\varepsilon}(y)}{\varepsilon} \right)\,d\nu(y)\quad \forall x\in\Omega_{1}.
\end{align*}
Therefore, using some elementary inequalities in the last display, for every $x\in \Omega_{1}'$, we see 
\begin{align*}
    |\nabla^2u_{\varepsilon}(x)| 
    &\leqslant
    \frac{2}{\varepsilon} \I_{\Omega_{2}}  \left(|y-\nabla u_{0}(x)|^2 + |\nabla u_{0}(x)-\nabla u_{\varepsilon}(x)|^2\right)\exp\left( \frac{\inner{x}{y}-u_{\varepsilon}(x)-v_{\varepsilon}(y)}{\varepsilon} \right)\,d\nu(y)
    \notag\\&
    \leqslant
     \frac{c}{\varepsilon^{1-2b}} + \frac{2}{\varepsilon}\I_{\Omega_{2}}|y-\nabla u_{0}(x)|^{2} \exp\left( \frac{\inner{x}{y}-u_{\varepsilon}(x)-v_{\varepsilon}(y)}{\varepsilon} \right)\,d\nu(y)
\end{align*}
for a constant $c\equiv c(n,\lambda,\Lambda,\Omega_{1}')$, where we have used the estimate \eqref{lem:ABF:1}. Arguing similarly as in the proof of Lemma \ref{lem:1E} and using \eqref{lem:ABF:1}, we conclude with the inequality \eqref{lem:ABF:2}. 
\end{proof}

\section{Proof of Theorem \ref{thm:main}}

Finally, we provide the proof for the main result of this paper, namely Theorem \ref{thm:main}. Let $\Omega_{1}'\ssubset \Omega_{1}$ be a fixed convex compact subset and let $x,y\in \Omega_{1}'$ be arbitrary points. Let us consider the following two cases.\vspace{2mm}

\noindent
\textbf{Case 1: $|x-y|>\varepsilon$.} By the triangle inequality, Lemma \ref{lem:1E}
 and Theorem \ref{thm:Caf}, we have
\begin{align*}
    |\nabla u_{\varepsilon}(x)-\nabla u_{\varepsilon}(y)| 
    &\leqslant
    |\nabla u_{\varepsilon}(x)-\nabla u_{0}(x)| + |\nabla u_{0}(x)-\nabla u_{0}(y)| + |\nabla u_{\varepsilon}(y)-\nabla u_{0}(y)|
    \notag\\&
    \leqslant
    2\norm{\nabla u_{\varepsilon}-\nabla u_{0}}_{L^{\infty}(\Omega_{1}')} + |\nabla u_{0}(x)-\nabla u_{0}(y)|
    \notag\\&
    \leqslant
    2B\varepsilon^{b} + c|x-y|^{\alpha}
    \leqslant
    c_{*}|x-y|^{\min\{b,\alpha\}}
\end{align*}
for a constant $c_{*}\equiv c_{*}(n,\lambda,\Lambda,\Omega_{1}', \Omega_{1})$.\vspace{2mm}

\noindent
\textbf{Case 2: $|x-y|\leqslant \varepsilon$.} Using the mean value theorem, the convexity of $\Omega_{1}'$ and Lemma \ref{lem:ABF}, we have
\begin{align*}
    |\nabla u_{\varepsilon}(x)-\nabla u_{\varepsilon}(y)| 
    &\leqslant
    \norm{\nabla^2u_{\varepsilon}}_{L^{\infty}(\Omega_{1}')}|x-y|
    \notag\\&
    \leqslant
    \frac{c_{*}}{\varepsilon^{m}}|x-y|
    \notag\\&
    \leqslant
    c_{*}|x-y|^{1-m}
\end{align*}
for a constant $c_{*}\equiv c_{*}(n,\lambda,\Lambda,\Omega_{1}', \Omega_{1})$, thereby concluding the proof.

\subsection*{Acknowledgements}
S.B.\ and A.G.\ acknowledge the support of the Canada Research Chairs Program, the Natural Sciences and Engineering Research Council of Canada, and the Applied Mathematics laboratory of the \textit{Centre de recherches math\'ematiques}. A.G. also acknowledges travel support from the European Union's Horizon 2020 research and innovation programme under grant
agreement No 951847. S.B.\ also was supported by the Balzan project led by Luigi Ambrosio.  This project has received funding from Ministero dell'Universit\`a e della Ricerca (PRIN project 202244A7YL of S.D.M. and MIUR Excellence Department Project 2023-2027 awarded to the DIMA of the University of Genova, CUP D33C23001110001), Air Force (AFOSR project FA8655-22-1-7034 of S.D.M.) and Istituto Nazionale di Alta Matematica (S.D.M. is a member of GNAMPA). The final redaction of this paper was performed while A.G. \ was visiting the Institute for Pure and Applied Mathematics (IPAM), which is supported by the National Science Foundation (Grant No. DMS-1925919). S.D.M. wants to thank Guido De Philippis and Felix Otto for useful discussions about the topic of the paper.

\bibliographystyle{abbrv}
\bibliography{refs.bib}

\end{document}